\documentclass[12pt]{article}
\usepackage{graphicx}

\usepackage[fleqn]{amsmath}
\usepackage[paperwidth=8.5in, paperheight=11in,total={6.5in,8.5in}]{geometry}
\usepackage{amsthm}
\usepackage{amssymb,hyperref}
\usepackage{amsbsy,fancyhdr}
\usepackage{authblk}
\usepackage{amsfonts}
\usepackage{mathrsfs}
\usepackage{amsmath}
\usepackage[all]{xy}
\usepackage{amstext}
\usepackage{amscd}
\usepackage[dvips]{epsfig}
\usepackage{psfrag}
\usepackage{enumerate}
\usepackage{flafter}
\allowdisplaybreaks
\usepackage{color}
\usepackage{todonotes}
\usepackage{marvosym}

\usepackage{tikz}

\theoremstyle{plain}
\newtheorem{thm}{Theorem}[section]
\newtheorem{prop}[thm]{Proposition}
\newtheorem{cor}[thm]{Corollary}

\theoremstyle{definition}
\newtheorem{exa}[thm]{Example}

\newtheorem{rem}[thm]{Remark}
\newtheorem{defn}[thm]{Definition}

\newtheorem{construction}[thm]{Construction}

\def\Ker{\mathop{\mathrm{Ker}}\nolimits}

\def\Hom{\mathop{\mathrm{Hom}}\nolimits}

\newcommand{\lra}{\longrightarrow}
\newcommand{\ra}{\rightarrow}

\newcommand{\As}{{\rm As}}
\newcommand{\Inn}{{\rm Inn}}

\begin{document}

\title{{\bf Schur Multipliers and Second Quandle Homology}}

\author[\Coffeecup]{Rhea Palak Bakshi}
\author[\Coffeecup]{Dionne Ibarra}
\author[\Coffeecup]{Sujoy Mukherjee}
\author[$\ddag$]{Takefumi Nosaka}
\author[\Coffeecup, $\star$]{J\'{o}zef H. Przytycki}

\affil[\Coffeecup]{Department of Mathematics, The George Washington University}
\affil[$\ddag$]{Department of Mathematics, Tokyo Institute of Technology}
\affil[$\star$]{Department of Mathematics, University of Gda\'{n}sk}

\maketitle
\begin{abstract}
We define a map from second quandle homology to the Schur multiplier and examine its properties. Furthermore, we express the second homology of Alexander quandles in terms of exterior algebras. Additionally, we present a self-contained proof of its structure and provide some computational examples.
\end{abstract}

\begin{center}
\end{center}

\begin{center}
\end{center}

\tableofcontents

\section{Introduction }

\

A quandle \cite{Joy,Mat} is a set with a binary operation whose definition was motivated by knot theory. A primordial example of a quandle is a group 
with the binary operation given by conjugation. 
Rack (co)homology was introduced in \cite{FRS1,FRS3}.
It was adapted into a quandle (co)homology theory to define topological 
invariants for knots of codimension two \cite{CJKLS}. 

\

Since, there are adjoint functors between the category of quandles and the category of groups \cite{Joy}, it is natural to expect relations between the homology theories of quandles and groups. An example of such a relation is the analogy between 
central group extensions and 
quandle extensions \cite{Bro,BT,CENS}.
Since, there are well-developed methods to compute group homology \cite{Bro, BT, Kar}, 
to gain more information, maps are constructed from the homology of quandles to that of groups \cite{CEGS,EG,Kab,Nos7}. 
However, 
these correspondences are often not isomorphisms. On the other hand, it was observed that the second homology of a Takasaki quandle of odd order is isomorphic to the Schur multiplier\footnote{The subject was introduced by Schur in his work on the projective representations of finite groups \cite{Sch}. According to \cite{Kar}, the original definition of Schur multiplier $M(G)$ is the second cohomology with coefficients in $\mathbb{C}^{\times}$, the multiplicative group of invertible complex numbers. For a finite group $G$, we have $H_2(G,\mathbb{Z})=H^2(G,\mathbb{C}^{\times})$. The book by Beyl and Tappa \cite{BT} gives the definition of the Schur multiplicator, commonly known as the Schur multiplier, $M(G)$ via the Schur-Hopf formula and in this case, $M(G)$ is isomorphic to $H_2(G,\mathbb{Z})$.} of the underlying  abelian group, that is, the exterior square of the group \cite{Miller,NP}.

\

In this paper, we demonstrate a new relation between the second quandle homology and the second group homology. Section 2 reviews quandles and quandle homology.
In Section \ref{S3}, we construct a homomorphism $\mathcal{T}^*$ from the second quandle cohomology to the relative group cohomology by using diagram chasing techniques (Definition \ref{def4}). 
In Theorem \ref{thm2}, we discuss the bijection 
after localization
 and conclude Section \ref{S3} by giving an algorithm to get an explicit presentation of the homomorphism $\mathcal{T}^*$. 
In Section \ref{SQS}, we analyze the homomorphism when $X$ is an Alexander quandle and $(1-t)$ is invertible. In this case, we see a close relation between the second homology of $X$ and the exterior square $X \wedge X$. 
We conclude Section \ref{section4} by briefly discussing the homomorphism $\mathcal{T}^*$ for non-Alexander quandles. 
In Section \ref{SS5}, we give a  self-contained algebraic proof of Theorem \ref{prob1}. Finally, in the last section we discuss Alexander quandles which are connected but not quasigroups.

\ \\

\subsection*{Acknowledgments}
\

The fourth author was supported by a travel grant from MEXT, the Program for Promoting
the Enhancement of Research Universities, for his visit to George Washington University. The fifth author was partially supported by the Simons Foundation Collaboration Grant for Mathematicians - 316446 and CCAS Dean's Research Chair award.

\section{Preliminaries}
\label{S2}
\hspace*{25pt}In this section, we review basic quandle theory and the homology of groups and quandles.

\subsection{Quandles and their properties}
\label{SS21}
\hspace*{25pt}A {\it quandle} \cite{Joy,Mat} is a set, $X$, with a binary operation $* : X \times X \ra X$ such that:
\begin{enumerate}[(I)]
\item (idempotency) the identity $a* a=a $ holds for any $a \in X$ ,
\item (invertibility) the map $ (\bullet * b ): \ X \ra X$ defined by $x \mapsto x * b $ is bijective for any $b \in X$ and its inverse is denoted by  $\bullet * ^{-1} b$ ,
\item (right self distributivity) the identity $(a* b)* c=(a* c)* (b* c)$ holds for any $a,b,c \in X.$
\end{enumerate}

\

Consider an Abelian  group $X$ with an automorphism $t: X \rightarrow X$. Then, $X$ is a quandle with the operation $x* y:= t x+(1-t)y$, and is called {\it an Alexander quandle}. Notice that $X$ is also a $\mathbb{Z}[t^{\pm 1}]$-module. As a special case, if $t=-1$, the quandle is called {\it a Takasaki quandle} \cite{NP}.
In addition, every group $G $ has a quandle structure 
with operation $g* h:= h^{-1}gh $ and is called {\it a conjugation quandle}.

\

The group generated by the bijective maps $\bullet * x$ is called the {\it inner automorphism group} of $X$, and is denoted by $\mathrm{Inn}(X) $.
Observe that $\mathrm{Inn}(X) $ acts on $X$ from the right. Let $\mathrm{Stab}(x) $ be the stabilizer subgroup of $\mathrm{Inn}(X) $, for an element $x \in X$.
If the action is transitive, $X$ is said to be {\it connected}.
Furthermore, we define {\it the associated group of $X$}, $\As(X)$, by the presentation:
$$ \langle \ e_x \ \ (x \in X) \ | \ e_x e_y= e_y e_{x* y} \ \ (x,y \in X)\ \rangle. $$

\

In general, it is hard to concretely determine $\As(X)$; see \cite{Cla2} for the case of Alexander quandles. Analogous to Inn(X), As(X) acts on $X$ from the right as follows, $ e_y(x) = x*y . $ We denote this action by $X \curvearrowleft \mathrm{As}(X)$. Additionally, let $\widetilde{\mathrm{Stab}}(x_i) \subset \mathrm{As}(X) $ be the stabilizer subgroup and $\iota_i $ denote the inclusion $\widetilde{\mathrm{Stab}}(x_i) \subset \mathrm{As}(X) $.

\

Consider the commutative diagram in \eqref{33aa}.
Let $\psi : \As(X) \ra \Inn(X)$ be the homomorphism which sends $e_x$ to $ \bullet * x$. Then, we have the following group extensions:
\begin{equation}\label{33aa}{\normalsize
\xymatrix{
0 \ar[r] &
\Ker(\psi)\ar@{=}[d] \ar[rr]& & \As(X) \ar[rr]^{\psi } & & \Inn(X)\ar[r] & 0
& {\rm (exact )}\\
0 \ar[r] &
\Ker(\psi) \ar[rr] & & \widetilde{\mathrm{Stab}}(x_i) \ar[rr]^{\mathrm{res} \psi } \ar@{^{(}-_>}[u]_{ \iota_i }& & \mathrm{Stab}(x_i) \ar@{^{(}-_>}[u]_{\rm inclusion } \ar[r] & 0 & {\rm (exact )}.
}}\end{equation}
The horizontal sequences have been proven to be central extensions (see, for example, \cite[\S 2.3]{Nosbook}).

\subsection{Quandle Homology and Relative Group Homology}\label{SS22}

\

We now review the homology theory of quandles and groups.
Let $C_n^R(X)$ be the free $\mathbb{Z}$-module generated by $X^n$, i.e., $C_n^R(X) = \mathbb{Z} \langle X^n \rangle $.
For $n \leq 3$, the differential $\partial_n^R : C_n^R(X) \ra C_{n-1}^R(X) $ is defined by
\[ \partial_1^R(x):=0, \ \ \ \ \ \partial_2^R(x,y):= (x)-(x* y), \]
\begin{equation}\label{33aa2} \partial_3^R(x,y,z):= (x,z)-(x* y, z )-(x,y) +(x* z, y* z). \end{equation}


Then, the second homology groups are given by 
\begin{equation}\label{exa9}
H_2^R(X;\mathbb{Z}) := \frac{\Ker ( \partial_2^R )}{ \mathrm{Im} ( \partial_3^R ) }, \ \ \ \ \ \ \ \ H_2^Q(X;\mathbb{Z}) := \frac{\Ker ( \partial_2^R )}{ \mathrm{Im} ( \partial_3^R ) , \  \{ (a,a)\}_{a \in X } } .
\end{equation}
The former is called (two term) {\it rack homology} \cite{FRS1,FRS3}, and the latter is called {\it quandle homology} \cite{CJKLS}.
Dually, for an abelian group $A$, we can define quandle cohomology $H_Q^2(X;A)$ (see \cite{CJKLS} for details of general (co)homology).

\

\hspace*{5pt}
Next, we review relative group (co)homology (see \cite{BE,Bro}).
Let $A$ be a right $\mathbb{Z}[G]$-module.
Let $ C_{n}^{\rm gr }(G ;A ) $ be $A \otimes_{\mathbb{Z}[G]} \mathbb{Z}[G^n] $.
Define the boundary map $\partial_n( a \otimes (g_1, \dots , g_n)) \in C_{n-1}^{\mathrm{gr}}(G;A)$ by the formula:
$$ a \otimes ( g_2, \dots ,g_{n}) +\!\!\sum_{1 \leq i \leq n-1}\!\! (-1)^i a \otimes ( g_1, \dots ,g_{i-1}, g_{i} g_{i+1}, g_{i+2},\dots , g_n) +(-1)^{n} (a \cdot g_n^{-1})\otimes ( g_1, \dots, g_{n-1}) .$$
Since $\partial_{n-1} \circ \partial_n =0$, we can define the homology $H_{n}^{\rm gr }(G ;A) $ in the usual way.
 As mentioned before, if $A$ is the trivial $\mathbb{Z}[G]$-module, the second homology is often referred to as the {\it Schur multiplier} (see \cite{Bro,BT,Kar} for basic references).

\

Let $K \subset G$ be groups and set up the cochain groups as
$$C^n_{\rm gr}( G, K; A ):= \mathrm{Map} ( G^n , A ) \oplus \mathrm{Map }((K)^{n-1} , A ) . $$
For $(h,k ) \in C^n_{\rm gr}( G, K;A ) $, define
        $ \partial^n(h,k ): G^{n+1} \times K^n \rightarrow A $ 
        by the formula
$$ \partial^n \bigl(h,k \bigr)( a, b )= \bigl( h( \partial_{n+1} (a)), \ h (b) -k(\partial_n(b)) \bigr),$$
where $( a, b ) \in G^{n+1} \times K^{n} $. Then, we have a cochain complex $ (C^*_{\rm gr}( G, K ; A ), \partial^*)$, and we define the cohomology in the standard way. Since $C^n_{\rm gr}( G, K; A ) $ is defined as a mapping cone, we have the long exact sequence:
\begin{equation}\label{exa1}
{\cdots \ra H_{\rm gr }^{n-1}(K;A) \stackrel{\delta_{n-1}^*}{\ra} H^n_{\rm gr}(G,K;A) \ra H^n_{\rm gr}(G;A) \ra H^n_{\rm gr}(K;A) \stackrel{\delta_{n}^*}{\ra} H^{n+1}_{\rm gr}(G;A) \ra \cdots.}
 \end{equation}

\

\hspace*{5pt}Furthermore, for any central extension $ 0 \ra K \ra G \ra \mathcal{G} \ra 0$ and any trivial coefficient $A $, we recall the exact sequence
\begin{equation}\label{exa2} 0 \ra H^1_{\rm gr}(\mathcal{G};A ) \ra H^1_{\rm gr}(G;A ) \stackrel{\delta^*_1}{ \ra} \Hom(K,A )\ra H^2_{\rm gr} (\mathcal{G} ;A ) \ra  H^2_{\rm gr}(G;A ) , \end{equation}
which is called {\it the 5-terms exact sequence} (see, e.g., \cite[Chapter II.5]{Bro}).

\

\hspace*{5pt}
We now recall some known results in quandle homology. 
Let $X$ be a quandle.
Due to the action $X \curvearrowleft \Inn(X)$, we have the orbit decomposition $X= \sqcup_{i\in I} X_i$, where $I$ is the set of the orbits.
For $i \in I$, we choose $x_i \in X_i$.
Consider the induced map $(\iota_i)_*: H_1(\mathrm{Stab}(x_i);\mathbb{Z})\ra H_1(\As(X);\mathbb{Z})$ from the inclusion $\iota_i $ and the homomorphism $ \epsilon_X : \As(X)\ra \mathbb{Z} $ which sends $e_x$ to $1$.
Notice that, since $e_{x_i } \in \widetilde{\mathrm{Stab}}(x_i ) $, the induced map
$$ \iota_* : H_1^{\rm gr}(\widetilde{\mathrm{Stab}}(x_i) ) \lra \mathrm{Im}( (\iota_i)_*) =\mathbb{Z} $$
splits. Thus, we can fix the projection
\begin{equation}\label{lp} \mathcal{P}_i: H_1^{\rm gr}(\widetilde{\mathrm{Stab}}(x_i) ;\mathbb{Z} ) \lra \Ker( (\iota_i)_* ) \subset H_1^{\rm gr}(\widetilde{\mathrm{Stab}}(x_i) ;\mathbb{Z}) .\end{equation}
Then, Eisermann \cite{Eis3} proved the following result about second rack and quandle homologies:
\begin{thm}[{\cite{Eis3}}]\label{rthm1} 
There are isomorphisms:
\begin{equation}\label{33aab} H_2^R(X) \cong \bigoplus_{i \in I } H_1(\widetilde{\mathrm{Stab}}(x_i) ;\mathbb{Z} ) \ \ \ \ \ \ \ and \ \ \ \ \ \ \ H_2^Q (X;\mathbb{Z}) \cong \bigoplus_{i \in I } \Ker((\iota_i)_*). \end{equation}
Moreover, the map $ H_2^R (X;\mathbb{Z}) \ra H_2^Q (X;\mathbb{Z})$ from the projection in \eqref{exa9} is equal to the direct sum of maps described in equation \eqref{lp}, that is, $\oplus_{i \in I} \mathcal{P}_i $.
\end{thm}


\

\hspace*{5pt}
Let us define the type of $X$ by
$$ \mathrm{Type}(X)= \mathrm{min}\{ \ n \ | \ \ ( \cdots (x \overbrace{* y )\cdots ) * y}^{n\textrm{-times}} =x \ \
\mathrm{for \ any\ } x, y \in X \ \
\} \in \mathbb{N} \cup \{ \infty \}. $$
Regarding the second group homologies of $\As(X)$, the fourth author showed the following:
\begin{thm}[{\cite{Nos7}}]\label{rthm2}Let $X$ be a connected quandle. Then, $H_2^{\rm gr}(\As(X);\mathbb{Z} ) $ is annihilated by $\mathrm{Type}(X) $.
In particular, by the 5-terms exact sequence \eqref{exa2}, there is an isomorphism $ \Ker (\psi )_{(\ell)} \cong H_2^{\rm gr}(\Inn(X) ;\mathbb{Z} )_{(\ell)} \oplus \mathbb{Z}_{(\ell )}$ for any prime $\ell$ with $(\ell, \mathrm{Type}(X))=1$.
\end{thm}

\
\

In addition, regarding the rack second cohomology, we recall the following result by Etingof and Gra\~{n}a:
\begin{thm}[{\cite{EG}}]\label{rthm3} For any quandle $X$ and any abelian group $A$, there is an isomorphism
$$H^2_{R}( X;A) \cong H^1_{\rm gr} (\As (X); \ \mathrm{Map}(X,A)), $$
where $\mathrm{Map}(X,A)$ is regarded as a right $\mathbb{Z}[\As(X)]$-module.
\end{thm}
\noindent
Although this isomorphism is elegant, it does not give much insight on how to compute the rack cohomology. In fact, there are only a limited numbers of examples obtained
from Theorem \ref{rthm3}.

\section{Relation of quandle homology to Schur multipliers} 
\label{S3}

\

\hspace*{5pt}In this section, we will define a map that gives a connection between the second quandle cohomology and Schur multipliers. Let $X$ be a quandle, and $A$ a trivial coefficient module. Recall the orbit decomposition $X= \sqcup_{i\in I} X_i$, and choose $x_i \in X_i$, as in the previous section.

For $i\in I$, consider the following commutative diagram:
\begin{equation}\label{33aad}{ \footnotesize
\xymatrix{
0 \ar[d]& & & & & \\
H^1_{\rm gr } ( \mathrm{Stab}(x_i);A ) \ar[d] & & & & & \\
H^2_{\rm gr } (\mathrm{Inn}(X), \mathrm{Stab}(x_i) ;A ) \ar[d]_{p^* } & & & & & \\
H^2_{\rm gr } (\mathrm{Inn}(X);A )\ar[d]_{\rm inclusion^* } &
\Hom ( \Ker(\psi), A) \ar@{=}[d] \ar[l]_{\delta_1^* }& H^1_{\rm gr } (\As(X);A)\ar[l] \ar[d]_{\rm \iota_i^* } & H^1_{\rm gr } ( \Inn(X);A)\ar[l] \ar[d]_{\rm inclusion^* } & 0 \ar[l]
\\
H^2_{\rm gr } (\mathrm{Stab}(x_i);A) &
\Hom ( \Ker(\psi), A) \ar[l]_{\delta_1^* } & H^1_{\rm gr } (\widetilde{\mathrm{Stab}}(x_i);A) \ar[l]_{\mathrm{res} \psi ^* } & H^1_{\rm gr } (\mathrm{Stab}(x_i);A) \ar[l] & 0.  \ar[l]
}}\end{equation}
Here, the horizontal sequences arise from the 5-terms exact sequences, and the left vertical sequence is obtained from the long exact sequence \eqref{exa1}. 

\begin{defn}\label{def4}
Using the above diagram, we define the homomorphism
\begin{equation}\label{Tdef} \mathcal{T}_i : H^1_{\rm gr} (\widetilde{\mathrm{Stab}}(x_i) ; A ) \lra \frac{ H^2_{\rm gr} ( \mathrm{Inn}(X), \ \mathrm{Stab}(x_i) ;A ) }{ H^1_{\rm gr}( \mathrm{Stab}(x_i) ; A )} \end{equation}
by setting
$$ \mathcal{T}_i ( \phi) := (p^*)^{-1} (\delta_1^* \circ \mathrm{res} \psi^* \circ \mathcal{P}^*_i( \phi)) .$$
Here, $\mathcal{P}_i^*$ is a dual of \eqref{lp}. By diagram chasing, we can easily verify that $\mathcal{T}_i$ is well-defined.
\end{defn}
\begin{rem}\label{rs} By definition the kernel $\Ker( \mathcal{T}_i) $ contains the image $\iota^*_i  (H^1_{\rm gr} ( \As(X); A ) )$. By Theorem \ref{rthm1}, the domain of the sum $\oplus_{i \in I} \mathcal{T}_i$ can be replaced by $H^2_Q(X;A )$. Namely,
\begin{equation}\label{kkkk}
\bigoplus_{i \in I} \mathcal{T}_i : H^Q_2(X;A ) \lra \bigoplus_{i\in I } \frac{H^2_{\rm gr} ( \mathrm{Inn}(X), \ \mathrm{Stab}(x_i) ;A ) }{ H^1_{\rm gr}( \mathrm{Stab}(x_i) ; A )} . 
\end{equation} 
\end{rem}

\

In the connected case, we have the following results pertaining to the homomorphism $ \mathcal{T}_i = \mathcal{T}$.
\begin{thm}\label{thm2}
Let $X$ be a connected quandle, and fix $x_0 \in X.$
Let $\ell$ be a prime number which is coprime to $\mathrm{Type}(X)$.
Suppose that the localization $ H^1_{\rm gr} ( \mathrm{Stab}(x_0); A )_{(\ell)}$ is zero. 
Then, the localized map $\mathcal{T}_{(\ell)}$ gives an isomorphism
$$\mathcal{T}_{(\ell)}: H^2_Q ( X;A)_{(\ell)}\stackrel{\sim}{\lra} H^2_{\rm gr} ( \mathrm{Inn}(X), \ \mathrm{Stab}(x_0) ;A )_{(\ell)} . $$
\end{thm}
\begin{proof}
From Theorem \ref{rthm2} it follows that the localized map $(\delta_1^*)_{(\ell)}$ is
surjective and the kernel is equal to $\iota^*( H^1_{\rm gr}(\As(X);A)_{(\ell)}).$
Since $ H^1_{\rm gr} ( \mathrm{Stab}(x_0); A )_{(\ell)} =0$, the localized $ \mathcal{T}_{(\ell)} $ is surjective, and
the kernel is also equal to $\iota^* ( H^1_{\rm gr}(\As(X);A)_{(\ell)})$. Hence, bijectivity is established.
\end{proof}

\

\hspace*{5pt}A dual reconsideration of the above proof results in the following theorem. 
\begin{thm}\label{thm3} Let $X$ be a connected quandle. Using the same notation as in Theorem \ref{thm2}, we get the homomorphism:
\begin{equation}\label{bbb} \mathcal{T}_* :
\Ker \Bigl( H_2^{\rm gr} ( \mathrm{Inn}(X), \ \mathrm{Stab}(x_0) ;\mathbb{Z} ) \ra H_1^{\rm gr}( \mathrm{Stab}(x_0) ; \mathbb{Z} ) \Bigr) \lra
H_2^{Q} (X ; \mathbb{Z} ), \end{equation}
such that the map localized at $\ell$ yields the following isomorphism:
\begin{equation}\label{bbb4} \mathcal{T}_* :
H_2^{\rm gr} ( \mathrm{Inn}(X), \ \mathrm{Stab}(x_0) ;\mathbb{Z} )_{(\ell)} \stackrel{\sim }{\lra}
H_2^{Q} (X ; \mathbb{Z} )_{(\ell)}, \end{equation}
 provided that $\ell$ is a prime number, $\ell$ and $\mathrm{Type}(X)$ are relatively prime, and $ H_1^{\rm gr} ( \mathrm{Stab}(x_0) ; \mathbb{Z} )_{(\ell)}$ vanishes.
\end{thm}

\

We now emphasize the advantages of the homomorphism $ \mathcal{T} $. 
While it is true that Theorem \ref{rthm1} gives a way to compute second quandle (co)homology, in general, it is hard to definitively determine $\As(X)$ and $\Ker (\psi)$.
In contrast, according to Theorems \ref{thm2} and \ref{thm3},
we do not need any information of $\As(X)$ and $\Ker (\psi)$. What we do need is $\Inn(X)$ and its second group homology.
Furthermore, the map $ \mathcal{T}_i $ can be concretely described as follows.


\

For this, by the definitions, we shall only examine $\delta_1^* $ and the isomorphisms in \eqref{33aab} in detail. For a general central extension $0 \ra K \ra G \ra \mathcal{G} \ra 0$, we give a concrete description of the delta map $\delta_1^* : H^1_{\rm gr} ( K;A) \ra H^2_{\rm gr} ( \mathcal{G};A) $.
For $\phi \in H^1_{\rm gr}(K,A)$, choose a representative 1-cocycle $F: K \ra A $ with $[F] = \phi$, and a section $\mathfrak{s}: \mathcal{G} \ra G$.
We define a map
$$\delta( F ): \mathcal{G}^2 \lra A \ \ \mathrm{ by} \ \ \ \ \delta( F ) (g,h):= F(\mathfrak{s} (g) \mathfrak{s}(h)\mathfrak{s}(gh)^{-1})$$
for any $g,h \in \mathcal{G}$.
Then, $\delta^*_1(\phi )$ is equal to $[\delta( F) ] $ (see \cite{Rou} for a proof).

\

We now explain the isomorphism in \eqref{33aab}.
For $ i \in I$ and any element $ [g] \in \widetilde{\rm Stab}(x_i)$, we choose a representative $g = e_{x_1}^{\epsilon_1} \cdots e_{x_n}^{\epsilon_n}$ for some $x_1, \dots, x_n \in X$ and $\epsilon_1,\dots, \epsilon_n \in \{ \pm 1\} $. Now, consider the correspondence
$$\widetilde{\rm Stab}(x_i) \ni e_{x_1}^{\epsilon_1} \cdots e_{x_n}^{\epsilon_n} \longmapsto \sum_{j=1}^n \epsilon_i \Bigl(\bigl( ( \cdots ( x_i * ^{\epsilon_1} x_1 ) * ^{\epsilon_2} \cdots ) * ^{\epsilon_{j-1}} x_{j-1}\bigr) *^{\frac{\epsilon_j -1 }{2}} x_j , \ x_{j} \Bigr)\in Z_2^R(X). $$
Since this correspondence is additive, we have an induced map $ H_1^{\rm gr}(\widetilde{\rm Stab}(x_0) ;\mathbb{Z} )
\ra H_2^R(X)$. This map is exactly equal to the isomorphism in \eqref{33aab} (see also \cite[\S 5.2]{Nosbook} for the proof).

\

When $X$ is of finite order, there are methods to determine the finite group $\mathrm{Inn}(X) $ (see for example, \cite{EMR} and \cite[Appendix B]{Nosbook}) and it is not so hard to choose a section $\mathrm{Inn}(X) \ra \As(X)$.
In summary, if we find an explicit representative of a quandle 2-cocycle $\phi: X^2 \ra A$, we can compute $\mathcal{T}(\phi). $

\section{The second homology of Alexander quandles}\label{section4}
Based on the computation of $ \mathcal{T}$ in the previous section,
we will concretely describe $\mathcal{T}$ where $X$ is an Alexander quandle of type $\mathrm{Type}(X)$.
That is, $X$ is a $\mathbb{Z}[t^{\pm 1}]$-module, and the quandle operation is defined by $x* y= tx+(1-t)y$. Recall that $\mathrm{Type}(X)$ is equal to ${\mathrm{min}\{ n \in \mathbb{N} \ | \ t^n= \mathrm{id}_X \ \}}$ if it exists or $\infty $ otherwise, 
and $X$ is connected if and only if $1 -t :X \ra X$ is surjective.

We now analyze $\As(X)$, when $ (1-t)$ is invertible.
Consider the homomorphism $X \otimes_{\mathbb{Z}} X \ra X \otimes_{\mathbb{Z}} X $ which takes $x\otimes y $ to $x \otimes y - y \otimes t x$, and let $Q_X$ be its cokernel. 
Recall that Clauwens \cite{Cla2} showed a group isomorphism 
$$ \As(X) \lra \mathbb{Z} \times X \times Q_X ,$$
where the group operation on the right hand side is defined by
\begin{equation}\label{abc3}(n,a,\alpha ) \cdot (m, b, \beta)=(n+m, t^{m} a+b, [ a \otimes t^{m}b] + \alpha + \beta ). \end{equation}
\subsection{Second quandle homology from Schur multipliers}\label{SQS}
In Section \ref{SS5}, we give a self-contained proof of the following theorem. 
\begin{thm}\label{prob1}
Let $X$ be an Alexander quandle with $(1-t)$ invertible. 
The correspondence $\mathbb{Z} \langle X^2 \rangle \ni (x,y) \mapsto x \otimes (1-t)y \in X \otimes X $ gives rise to the isomorphism 
\begin{equation}\label{abc2} H_2^Q(X;\mathbb{Z}) \cong \frac{X\otimes_{\mathbb{Z}} X}{ \{ x \otimes y - y \otimes t x \}_{x,y \in X }}.\end{equation}

\end{thm}
\begin{rem}\label{kk} The existence of the isomorphism in \eqref{abc2} is implicit in \cite{Cla2,IV}. In fact, since we can easily check that $ \Ker (\iota_i) \cong Q_X$ from the group operation in \eqref{abc3}, 
then the isomorphism $H_2^Q(X) \cong \Ker (\iota_i) $ in Theorem \ref{rthm1} readily implies the isomorphism in \eqref{abc2}.
However, Theorem \ref{rthm1} uses topological properties of the rack space \cite{FRS1,FRS3}, where the rack space is a geometrical realization of the rack complex. In contrast, Section \ref{SS5} will give an independent algebraic proof of Theorem \ref{prob1} and a concrete description of the isomorphism.
\end{rem}

To describe the homomorphism $\mathcal{T}$,
we consider the following transformation:
\begin{cor}\label{prob7}
Let $X$ be an Alexander quandle with $(1-t)$ invertible.
Then, there is an isomorphism
\begin{equation}\label{abc} H_2^Q(X;\mathbb{Z}) \cong \frac{X\wedge_{\mathbb{Z}} X}{ \{ x \wedge y - tx \wedge ty \}_{x,y \in X }} . \end{equation}
\end{cor}
\begin{proof} We can easily check that the correspondence $ X \otimes X \ra X \otimes X$ which sends  $x \otimes y$ to $x \otimes (1-t)y,$ defines a homomorphism from the right hand side of \eqref{abc2} to that of \eqref{abc}.
Moreover, the inverse mapping is obtained from the correspondence $ X \otimes X \ra X \otimes X$ which sends $x \otimes y$ to $x \otimes (1-t)^{-1}y$. 
\end{proof}
\begin{rem}\label{NPrem} If $t=-1$, the isomorphism in \eqref{abc} has been proved in \cite{NP}. Therefore, Corollary \ref{prob7} is a generalization of \cite{NP}.
\end{rem}


Next, we explicitly describe the homomorphism $ \mathcal{T} $.
It is not hard to check that $\Inn(X) \cong (\mathbb{Z}/ \mathrm{Type}(X)) \ltimes X$ (see for example, \cite[Proposition B.18]{Nosbook}).
In order to describe $\delta^*_1 $ in detail, we choose to define a section $\mathfrak{s}: \Inn(X) \ra \As(X) $ by $ ([n],x) \mapsto (n,x,0)$. 
The universal quandle 2-cocycle is represented by (see Corollary \ref{cor8})
\begin{equation}\label{abc4} \phi: X \times X \lra Q_X ; \ \ \ \ \ \ (x,y) \longmapsto [x\otimes (1-t)y ]. \end{equation}
Then, according to the discussion in \S \ref{S3}, we can easily show that the 2-cocycle of the group $\mathrm{Inn}(X)$ is given by the map:
\begin{equation}\label{abc5} \mathcal{T} (\phi):\mathrm{Inn}(X)^2 \lra \frac{X\wedge_{\mathbb{Z}} X}{ \{ x \wedge y - tx \wedge ty \}_{x,y \in X }} ; \ \ \ \ \ \ ((n,x),(m,y)) \longmapsto x\wedge y .\end{equation}

Finally, we will compute the domain of $\mathcal{T}_*$, i.e., the left hand side of \eqref{bbb}, and check Theorem \ref{thm3} when $X$ is an Alexander quandle and $(1-t)$ is invertible.
Since $\Inn(X) \cong (\mathbb{Z}/ \mathrm{Type}(X)) \ltimes X$ as above, $ \mathrm{Stab}(x_0) \cong \mathbb{Z}/ \mathrm{Type}(X) \mathbb{Z} $. Therefore, $H_k(\mathrm{Stab}(x_0) ;\mathbb{Z})$ is annihilated by $ \mathrm{Type}(X) $.
Hence, the left hand side is isomorphic to $ H_2^{\rm gr} ( \mathrm{Inn}(X);\mathbb{Z})$ after localization at $\ell $.
Then, by using transfer (see for example, \cite[Sect. III.8]{Bro}), we have the following isomorphisms:
$$H_2^{\rm gr} ( \mathrm{Inn}(X);\mathbb{Z})_{(\ell)} \cong \Bigl( \frac{H_2^{\rm gr} ( X;\mathbb{Z} ) }{ \{ t a -a \}_{ a \in H_2^{\rm gr} ( X;\mathbb{Z} ) }} \Bigr)_{(\ell )} \cong \Bigl( \frac{X\wedge_{\mathbb{Z}} X}{ \{ x \wedge y - tx \wedge ty \}_{x,y \in X }} \Bigr)_{(\ell )} . $$
Here, the second isomorphism is obtained from $ H_2^{\rm gr} ( X;\mathbb{Z} ) \cong X\wedge_{\mathbb{Z}} X$ and using the fact that the action of $\mathbb{Z}/ \mathrm{Type}(X) $ on $H_2^{\rm gr}(X;\mathbb{Z})$ is compatible with the diagonal action on $X\wedge_{\mathbb{Z}} X$ (see \cite[Chapter V.6]{Bro}).
In particular, after comparing \eqref{abc4} with \eqref{abc5}, we see that the localized $\mathcal{T}_{(\ell)}$ is an isomorphism as in Theorem \ref{thm3}.



\subsection{Examples of computations of the second homology of Alexander quandles}\label{SS4}

In this subsection, we compute the second homology of certain Alexander quandles using Corollary \ref{prob7}. 
The homomorphism $T : X\wedge_{\mathbb{Z}} X \ra X\wedge_{\mathbb{Z}} X $ which sends $x \wedge y$ to $ t x \wedge t y $ plays a key role.


First, we compute the second homology of a connected quandle of order $p^2$.
Although the result is first proven by \cite{IV}, we give a simpler proof.
To describe this, we say that a connected Alexander quandle $Y$ is {\it special}, if
$Y \cong (\mathbb{Z}/p)^2 $ and the determinant of $t: Y \ra Y $ is 1.
\begin{prop}[{\cite[Proposition 5.9]{IV}}]\label{prob9}
Let $X$ be a connected quandle of order $p^2$.
If $X$ is a special Alexander quandle, then $ H_2^Q(X) \cong \mathbb{Z}/p$.
Otherwise, $ H_2^Q(X)$ is zero.
\end{prop}
\begin{proof} It is shown in \cite{EG} that $X$ is isomorphic to an Alexander quandle. 
Notice that, $X$ is either $\mathbb{Z}/p^2 \mathbb{Z}$ or $ (\mathbb{Z}/p\mathbb{Z} )^2 $.
If $ X = \mathbb{Z}/p^2 \mathbb{Z}$, then $X \wedge X=0$; hence, $H_2^Q(X)$ is zero.

Thus, we may suppose $ X \cong (\mathbb{Z}/p\mathbb{Z} )^2.$ Since $X \wedge X\cong \mathbb{Z}/p$, $H_2^Q(X)$ is zero or $\mathbb{Z}/p $.
Therefore, it is enough to show that $H_2^Q(X) \cong \mathbb{Z}/p $ if and only if $X$ is special.
Recall the homomorphism $T : X\wedge_{\mathbb{Z}} X \ra X\wedge_{\mathbb{Z}} X $, and that
$H_2^Q(X) \cong \mathrm{Coker} (1- T)$. 
We can easily get $T = \mathrm{det}(t) \cdot \mathrm{id}_{X \wedge X}$ by considering the eigenvalues of $t$.
Hence,
$H_2^Q(X) \cong \mathrm{Coker} (1- T)$ is $ \mathbb{Z}/p$, if and only if $ \mathrm{det}(t)=1 $, that is, $X$ is special.
\end{proof}

As another example, we will compute the second homology of the Alexander quandle constructed using the polynomial $\Phi_n= (t^n -1 )/(t-1)$.
\begin{prop}\label{prob5}
Let $X$ be an Alexander quandle of the form $X= \mathbb{F}_p [t]/( \Phi_n)$ over $\mathbb{F}_p$, where $\mathbb{F}_p$ denotes the field of order $p.$
Assume that $ n$ and $p$ are coprime.
Then, there is an isomorphism
\begin{equation}\label{abc34} H_2^Q(X;\mathbb{Z}) \cong (\mathbb{Z}/p)^{ \lfloor \frac{n-1}{2}\rfloor } .\end{equation}
\end{prop}
\begin{proof} By assumption, $X$ is connected and considered over a field $\mathbb{F}_p$.
Thus, $H_2^Q(X)$ is a vector space over $\mathbb{F}_p$ by Theorem \ref{prob1}. 
Thus, it is enough to show that $ \mathrm{dim} (\mathrm{Coker} (1-T))= \lfloor \frac{n-1}{2}\rfloor$.
Let $\overline{\mathbb{F}}_p$ be the algebraic closure of $ \mathbb{F}_p $,
and $X_{\overline{\mathbb{F}}_p} $ be $ \overline{\mathbb{F}}_p \otimes_{\mathbb{Z}/p} X$.
Fix $\mu \in \overline{\mathbb{F}}_p$ as an $n$-th root of unity.
Then, we have the extension
$$ \mathrm{id}_{\overline{\mathbb{F}}_p} \otimes T :
X_{\overline{\mathbb{F}}_p}\wedge_{\overline{\mathbb{F}}_p} X_{\overline{\mathbb{F}}_p} \lra X_{\overline{\mathbb{F}}_p} \wedge_{\overline{\mathbb{F}}_p}X_{\overline{\mathbb{F}}_p}
$$
and $ \mathrm{dim}( \mathrm{Coker} (1- T)) = \mathrm{dim}( \mathrm{Coker} ( 1- \mathrm{id}_{\overline{\mathbb{F}}_p} \otimes T))$. When considering $t : X_{\overline{\mathbb{F}}_p} \ra X_{\overline{\mathbb{F}}_p} $ as a linear map, we can choose the eigenvectors
$v_1 , \dots, v_{n-1} $ with $t v_i = \mu^{i} v_i$.
Since, dim$( X_{\overline{\mathbb{F}}_p})=n-1$, the vectors $v_1,\dots, v_{n-1}$ give a basis of  $X_{\overline{\mathbb{F}}_p}$.
Then, notice $ \mathrm{id}_{\overline{\mathbb{F}}_p} \otimes T ( \sum_{s,t}a_{s,t } v_s \wedge v_{ t}) = \sum_{s,t} a_{s,t } \mu^{s+t} v_s \wedge v_{ t}$, where $ a_{s,t}\in \overline{\mathbb{F}}_p $.
In particular, $T$ can be represented as a diagonal matrix and $ \mathrm{id}_{\overline{\mathbb{F}}_p} \otimes T ( v_j \wedge v_{ n-j}) = v_j \wedge v_{ n-j}$,
which implies that $ \mathrm{dim}( \mathrm{Coker} ( 1- \mathrm{id}_{\overline{\mathbb{F}}_p} \otimes T)) = \lfloor \frac{n-1}{2}\rfloor$, as required. 
\end{proof}



\subsection{Proof of Theorem \ref{prob1}}\label{SS5}


A quandle $X$ is a {\it quasigroup} if there exists a unique $c$ such that $a * c= b$ for any $a,b \in X$. We denote by $a \circ b$ such an element $c$. Observe that if $X$ is a quasigroup, then $X$ is connected.
In particular, an Alexander quandle $X$ is a quasigroup if and only if $(1-t)$ is invertible. We need the following basic properties proven in [NP](see Lemma 2.1, Corollary. 2.2). 
\begin{prop}[\cite{NP}]\label{lem5}
Let $X$ be a quasigroup quandle.
Choose any element $d_0$ in $X$. Then,
\begin{enumerate}
[(I)]
\item{For any element $d_{0}$ in $X$, the quotient map $C^{Q}_{2} \lra C^{Q}_{2}/((d_0,b))$ is an isomorphism when restricted to $\Ker \partial_2^R $. Here, $(d_0,b)$ is the subgroup of $C^{Q}_{2}$ generated by elements of the form $(d_0,b)$ for any $b\in X$, }
\item
The induced map 
\begin{equation}\label{hhh} H_2^Q (X) \lra \frac{\mathbb{Z}[X \times X] }{\{ (a,a), \ (d_0,a) , \ \partial_3^R(a,b,c)\}_{a,b,c \in X} } ,\end{equation}
is an isomorphism.
\end{enumerate}
\end{prop}
Furthermore, we will also need the following proposition:
\begin{prop}\label{lem6}
Let $X$ be a quasigroup quandle.
Choose any element $d_0$ in $X$. In the quotient on the right hand side of \eqref{hhh}, the relation $(c, d_0)=0$ holds for any $c \in X$. 
\end{prop}
\begin{proof}
In the quotient, compute $\partial_3^R(d_0, a, d_0)$ as $ (d_0 ,d_0)-(d_0 * a, d_0) -( d_0,a ) +(d_0* d_0, a * d_0)= -(d_0 * a, d_0)$.
Since $X$ is quasigroup, $d_0 * a$ can be regarded as any element $c$. Therefore, we have $(c, d_0)=0$.
\end{proof}

Now, we give a self-contained proof of Theorem \ref{prob1}. Let $X$ be an Alexander quandle such that $(1-t)$ is invertible, and $d_0= 0$. 

In a quasigroup quandle we can write $[a, (1-t ) b]$ for $(a, b)$.
Then, the boundary map in \eqref{33aa2} has the form:
$$ \partial_3 (a,b,c) = (a,c)-(a* b, c)-(a,b)+(a*c,b*c)
$$
\begin{equation}\label{r1}=
[a, (1-t)c] - [a, (1-t) b] - [ ta +(1-t) b, (1-t)c ] + [ ta +(1-t) c, (1-t)(tb +(1-t) c)].
\end{equation}

\begin{proof}[Proof of Theorem \ref{prob1}]
The following isomorphism follows 
from Propositions \ref{lem5} and \ref{lem6}: 
\begin{equation}\label{eq20} H_2^Q(X;\mathbb{Z}) \cong \frac{ \mathbb{Z} \langle X \times X \rangle }{ \{ [a,a(1-t)], \ [a, 0], \ [0,a], \ \eqref{r1} \}_{a,b,c \in X } } .
\end{equation}
The correspondence $[a,b]\mapsto a\otimes b$ gives rise to
an epimorphism
$$ \frac{ \mathbb{Z} \langle X \times X \rangle }{ \{ [a,a(1-t)], \ [a, 0], \ [0,a], \ \eqref{r1} \}_{a,b,c \in X } } \lra \frac{X\otimes_{\mathbb{Z}} X}{ \{ x \otimes y - y \otimes t x \}_{x,y \in X }}. $$
The goal is to show that this epimorphism is an isomorphism. For this,
it is enough to show that $[, ]$ is bilinear and $[a,b]=[b, ta ]$.

First,
we use \eqref{r1} to show that the equality $[a,b]=[b, ta ]$ holds. 
If we replace $a$ by $x$, $b$ by $ (1-t)^{-1}y$, and $c$ by $(1-t)^{-1}z$, the relation \eqref{r1} in the right hand side of \eqref{eq20} is equivalent to
\begin{equation}\label{r2}
[x, z] - [x, y] - [ tx + y , z ] + [ tx +z, ty + (1-t) z] =0.
\end{equation}
For $ z=0$, $x=0$, and $y=-tx$, we obtain, respectively, 
\begin{equation}\label{r3}
[tx, ty] = [ x, y] ,
\end{equation}
\begin{equation}\label{r4}
[y,z] = [ z, ty+(1-t)z] ,
\end{equation}
\begin{equation}\label{r5}
[x ,-t x] = [x,z]+[tx+z, -t^2 x +(1-t)z].
\end{equation}
Note that by \eqref{r4} the relation in \eqref{r5} can be reduced to
\begin{equation}\label{r7} [x,-t x]=[x,z]+[-x,tx+z] .
\end{equation}
In addition, using \eqref{r4}, the relation in \eqref{r2} is equivalent to
\begin{equation}\label{r8}
0=[x, z] - [x, y] - [ t x + y , z ] + [ y+(t-1)x, tx + z] .
\end{equation}
After replacing $z$ with $y$ in \eqref{r7} we get
\begin{equation}\label{r72} [x,-t x]=[x,y]+[-x,tx+y] .
\end{equation}
Then, using \eqref{r7} and (\ref{r72}), we have
\begin{equation}\label{r9}
0= -[-x, tx+y] + [-x, tx+z] - [ t x + y , z ] + [ y+(t-1)x, tx + z].
\end{equation}
In addition, if we replace $(tx +y)$ by $Y$ and $(tx+z) $ by $Z$, then \eqref{r9} changes to 
\begin{equation}\label{r92}
0 = [-x, Z]-[-x, Y]+[Y,Z-tx]-[Y-x,Z].
\end{equation}
Then, after the substitution $Z=0$, the formula \eqref{r92} gives the desired equality $ [-x,Y]=[Y, -tx],$ for any $x,Y \in X$.


Finally, we show the bilinearlity of the bracket $[,]$. By the previous equality, it is enough to show that $ [y',Z ] +[x', Z]=[x'+y',Z]$ for any $x',y',Z \in X$.
Now, by applying the equality $ [ Y,Z-tx ]=[t^{-1}Z-x, Y]$ to \eqref{r92}, we have
\begin{equation}\label{r10}
[-x, Z]- [Y-x,Z] = [-x, Y]-[ t^{-1}Z -x , Y].
\end{equation}
Then, the substitution, $Z \mapsto Y$, $Y \mapsto t^{-1}Z$ transforms \eqref{r10} into
\begin{equation}\label{r102}
[-x, Y]- [t^{-1}Z -x ,Y] = [-x, t^{-1}Z]- [t^{-1}Y-x,t^{-1}Z].
\end{equation}
Hence, combining \eqref{r10} with \eqref{r102} yields the equation 
\begin{equation}\label{r11}
[-x, Z]- [Y-x,Z] = [-x, t^{-1}Z]- [t^{-1}Y-x,t^{-1}Z].
\end{equation}
In particular, when $ x=Y$, we have the equality $[-x,Z]=[-x, t^{-1}Z]-[ (t^{-1}-1)x, t^{-1} Z]$.
Thus, \eqref{r11} is reduced to
$$ [Y-x,Z] +[ (t^{-1}-1)x, t^{-1} Z] = [t^{-1}Y-x,t^{-1}Z] =[ Y -tx , Z],$$
where we use \eqref{r3} for the last equality.
Then, replacing $(Y-x)$ by $ y'$, and $(1-t)x$ by $ x'$, we have
$ [y',Z ] +[x', Z]=[x'+y',Z]$. 
\end{proof}
From the proof, we have the following corollary.
\begin{cor}\label{cor8} The map $X^2 \ra X \otimes X / \{ x \otimes y - y \otimes tx \}$, which sends $ (x,y) $ to $x \otimes (1-t)y$, is a quandle 2-cocycle.
\end{cor} 

\subsection{Non-Alexander quandles}\label{SS45}
Up until now, we mainly focused on Alexander quandles. In general, it is not easy to compute $H_2^Q(X)$ using Schur multipliers. 
In fact, we have verified that, for any connected non-Alexander quandle of order $\leq 15$, the map $\mathcal{T}$ in Definition \ref{def4} is zero.
Consider the following example. 
\begin{exa}\label{example1}
Let $\mathfrak{S}_n $ be the permutation group on $n$ elements, and
let $X$ be the subset $\{ g^{-1} (12) g\}_{g \in  \mathfrak{S}_n } \subset \mathfrak{S}_n $. Then, $X$ is of order $n(n-1)/2 $,
and the conjugacy operation makes $X$ into a connected quandle of type 2. 
Then, $ \mathrm{Inn}(X)\cong \mathfrak{S}_n $ and $\mathrm{Stab}(x_0) = \mathfrak{S}_{n-1} $. In particular, $H_1^{ \rm gr}( \mathrm{Stab}(x_0);\mathbb{Z}) \cong \mathbb{Z}/2\mathbb{Z}$. Moreover, it is known that $H_2^Q (X) \cong \mathbb{Z}/2\mathbb{Z}$ for $n>3$ (See \cite[Example 1.18]{Eis3}), but by diagram chasing in \eqref{33aad}, 
the map $\mathcal{T}$ turns out to be zero.
\end{exa}
Furthermore, there are few examples of connected quandles $ X$ such that the orders of $H_2^Q(X)$ and $\mathrm{Type}(X)$ are relatively prime. Thus, in the future one may consider finding other families of quandles for which Theorem \ref{thm3} is applicable.

\section{Alexander quandles and semi-Hopfian groups}
In this section, we give an example of an Alexander quandle which is connected but not a quasigroup. 

Analogous to the Hopfian condition\footnote{A group $X$ is called a Hopfian group if every epimorphism from $X$ to $X$ is an isomorphism.} for groups, we define a semi-Hopfian Abelian group as follows:

\begin{defn}
	An Abelian group $X$ is called semi-Hopfian if for every epimorphism $f: X \to X$,  such that $(1-f)$ is an 
	isomorphism, $f$ is also an isomorphism.
\end{defn} 

It is clear that every Abelian Hopfian group is semi-Hopfian. In particular, every finitely generated Abelian group is semi-Hopfian. Examples of non semi-Hopfian groups are well known \cite{Br}. For instance, consider the following construction: 

\begin{construction}\label{constr}
Consider the countable direct sum of the group $\mathbb Z$ indexed by positive numbers: 
	$$X=\bigoplus_{i>0} \mathbb{Z}^{(i)}$$
	where $1_i$ is the identity of $\mathbb{Z}^{(i)}$.\\
	Let $t:X\to X$ be an epimorphism defined on the basis by $f(1_i)=1_{i-1}$ for $i>1$ and $f(1_1)=0$.
	Clearly $f$ is not a monomorphism. We have $(1-f)(1_i)= t_i-t_{i-1}$ and observe that $(1-t)$ is invertible since:
	$$(1-t)^{-1}(1_i)=1_i+1_{i-1}+...+1_1.$$
\end{construction}

\

Construction \ref{constr} gives an example of an Alexander quandle which is connected but not a quasigroup. Namely, if we put $t=1-f$, then $(X,*)$, with $a*b=ta +(1-t)b$, is a quandle which is not a quasigroup since $1-t=f$  is not invertible. However, it is connected since $(1-t)X=X$.

\noindent
{\footnotesize
{\it Email address}: {\tt rhea\_palak@gwu.edu} \\
{\it Email address}: {\tt dfkunkel@gwu.edu} \\
{\it Email address}: {\tt sujoymukherjee@gwu.edu} \\
{\it Email address}: {\tt nosaka.t.aa@m.titech.ac.jp} \\
{\it Email address}: {\tt przytyck@gwu.edu}
}

\begin{thebibliography}{999999}

\small
\ifx\undefined\bysame
\newcommand{\bysame}{\leavevmode\hbox to3em{\hrulefill}\,}
\fi

\bibitem[BE]{BE}
R. Bieri, B. Eckmann, {\it Relative homology and Poincar\'{e} duality for group pairs},
J. Pure Appl. Algebra {\bf 13} (1978) 277-319.

\bibitem[Br]{Br}
G. Braun, Private Communication (December 3, 2018).

\bibitem[Bro]{Bro} K. S. Brown, {\it Cohomology of Groups}, Graduate Texts in Mathematics, { \bf 87}, Springer-Verlag, New York, 1994, x+306 pp.

\bibitem[BT]{BT}
 F. R. Beyl, J. Tappe, {\it Group extensions, representations, and the Schur multiplicator}, Lecture Notes in Mathematics, {\bf 958}, Springer-Verlag, Berlin-New York, 1982, iv+278 pp.


\bibitem[CEGS]{CEGS}
 J. S. Carter, M. Elhamdadi, M. Gra\~{n}a and M. Saito. {\it Cocycle knot invariants from quandle modules and generalized quandle homology}.
Osaka J. Math. {\bf 42} (2005) 499–-541.

\bibitem[CENS]{CENS} J. S. Carter, M. Elhamdadi, M. A. Nikiforou, M. Saito, {\it Extensions of quandles and cocycle knot invariants}.
 J. Knot Theory Ramifications 12 (2003), no. 6, 725-738.

\bibitem[CJKLS]{CJKLS}
J. S. Carter, D. Jelsovsky, S. Kamada, L. Langford, M. Saito,
{\it Quandle cohomology and state-sum invariants of knotted curves and surfaces},
Trans. Amer. Math. Soc. {\bf 355} (2003), 3947-3989.


\bibitem[Cla]{Cla2} F. J. - B. J. Clauwens, 
{\it The adjoint group of an Alexander quandle}, e-print: \href{https://arxiv.org/abs/1011.1587}{arXiv: 1011.1587 [math.GR]}. 


\bibitem[Eis]{Eis3} 
M. Eisermann, {\it Quandle coverings and their Galois correspondence}, Fund. Math. {\bf 225} (2014), 103-167.


\bibitem[EG]{EG} 
P. Etingof and M. Gra\~{n}a:
{\it On rack  cohomology}, J. of Pure and Applied Algebra,
{\bf 177} (2003), 49-59.

\bibitem[EMR]{EMR}
M. Elhamdadi, J. MacQuarrie, R. Restrepo, {\it Automorphism groups of quandles}, J. Algebra Appl., {\bf 11}
(2012), no. 1, 1250008, 9 pp.




\bibitem[FRS1]{FRS1} R. Fenn, C. Rourke, B. Sanderson, {\it Trunks and classifying spaces}, Applied Categorical Structures
, 3, 1995, 321-356.


\bibitem[FRS2]{FRS3} R. Fenn, C. Rourke, B. Sanderson, {\it The rack space}, Trans. Amer. Math. Soc. 359 (2007), no. 2, 701-740.


\bibitem[Gra]{Gra}
M. Gra\~{n}a. {\it Indecomposable racks of order $p^2$}, Beitr\"{a}ge Algebra Geom., 45(2): 665--676,
2004.

\bibitem[IV]{IV}
A. G. Iglesias, L. Vendramin, {\it An explicit description of the second cohomology group of a quandle}, 
Mathematische Zeitschrift, 
286 (2017), no. 3-4, 1041-1063.


\bibitem[Joy]{Joy}
D. Joyce,
{\it A classifying invariant of knots, the knot quandle},
J. Pure Appl. Algebra {\bf 23} (1982), 37-65.


\bibitem[Kab]{Kab}
Y. Kabaya, {\it Cyclic branched coverings of knots and quandle homology},
Pacific Journal of Mathematics,
{\bf 259} (2012), No. 2, 315-347.

\bibitem[Kar]{Kar}
G. Karpilovsky, {\it The Schur Multipliers}, Oxford Science Publications, London Mathematical society Monographs, New Series 2, Clarendon Press, Oxford, 1987,  x+302 pp. 




\bibitem[Mat]{Mat}
 S. Matveev, {\it Distributive groupoids in knot theory}, Matem. Sbornik
, 119(161)(1), 78-88, 1982 (in Russian); English Translation
in
Math. USSR-Sbornik
, 47(1), 73-83, 1984.


\bibitem[Mil]{Miller}
C. Miller, {\it The second homology of a group},
Proc. Amer. Math. Soc.
{\bf 3} (1952), 588-595.


\bibitem[Nos1]{Nos7}
T. Nosaka, {\it Central extensions of groups and adjoint groups of quandles}, 
 167–184, RIMS Kôkyûroku Bessatsu, B66, Res. Inst. Math. Sci. (RIMS), Kyoto, 2017.
 
\bibitem[Nos2]{Nosbook}
T. Nosaka, Quandles and Topological Pairs, Symmetry, Knots, and Cohomology,  Springer Briefs in Mathematics, Springer, Singapore, 2017, ix+136 pp.


\bibitem[NP]{NP}
M. Niebrzydowski, J. H. Przytycki, {\it The second quandle homology of the Takasaki quandle of an odd abelian group is an exterior
square of the group}. J. Knot Theory Ramifications
20 (2011), no. 1, 171--177, 	\href{https://arxiv.org/abs/arXiv:1006.0258}{eprint: arXiv:1006.025 [math.GT]}.




\bibitem[Rou]{Rou}
C. Rousseau, {\it D\'{e}formations d'actions de groupes et de certains r\'{e}seaux r\'{e}solubles}
Th\'{e}se de doctorat, Universit\'{e} de Valenciennes et du Hainaut-Cambr\'{e}sis, 2006.

\bibitem[Sch]{Sch}
J. Schur, {\it \"{U}ber die Darstellung der endlichen Gruppen durch gebrochen lineare Substitutionen}, J. Reine Angew. Math. 127 (1904), 20–-50.



\end{thebibliography}
\end{document}